\documentclass[11pt, normalheadings, oneside, reqno]{amsart}

\usepackage[english]{babel}

\usepackage[OT1]{fontenc}

\usepackage{mathtools}

\usepackage{amsmath,amssymb,amsfonts,mathrsfs}

\usepackage{bbm}

\usepackage[hidelinks]{hyperref}

\usepackage{enumitem}

\usepackage{bm}

\DeclareMathOperator\Mat{Mat}

\DeclareMathOperator\id{id}

\DeclareMathOperator\Conv{\overline{conv}}

\DeclareMathOperator\dist{dist}

\DeclareMathOperator{\card}{card}

\DeclareMathOperator{\Lip}{Lip}

\newtheorem{theorem}{Theorem}[section]
\newtheorem{question}{Question}[section]
\newtheorem{example}[theorem]{Example}
\newtheorem{definition}[theorem]{Definition}
\newtheorem{corollary}[theorem]{Corollary}
\newtheorem{lemma}[theorem]{Lemma}
\newtheorem{proposition}[theorem]{Proposition}

\theoremstyle{nonumberplain}

\newcommand{\R}{\mathbb{R}}

\newcommand*\norm[1]{\lVert#1\rVert}

\newcommand\abs[1]{\left\lvert#1\right\rvert}

\renewcommand{\epsilon}{\ensuremath\varepsilon}

\renewcommand{\phi}{\ensuremath{\varphi}}

\begin{document}
\title{Lipschitz extensions to finitely many points}
\author{Giuliano Basso}
\date{\today}
\maketitle
\begin{abstract}
We consider Lipschitz maps with values in quasi-metric spaces and extend 
such maps to finitely many points. We prove that in this context every 1-Lipschitz map admits an extension such that its
Lipschitz constant is bounded from above by the number of added
points plus one. Moreover, we prove that if the source space is a Hilbert space
and the target space is a Banach space, then there exists an
extension such that its Lipschitz constant is bounded from above by the square root of
the total of added points plus one. We discuss applications to metric transforms. 
\end{abstract}

\tableofcontents

\section{Introduction}
Lipschitz maps are generally considered as an indispensable tool in the study of metric spaces.
The need for a Lipschitz extension of a given Lipschitz map often presents itself naturally. 
Deep extension results have been obtained by Johnson, Lindenstrauss, and Schechtman \cite{johnson1986extensions}, Ball \cite{ball1992markov}, Lee and Naor \cite{lee2005extending}, and Lang and Schlichenmaier \cite{LangSchlichenmaier}. The literature surrounding Lipschitz extension problems is vast, for a recent monograph on the subject see \cite{brudnyi2011methods, brudnyimethods} and the references therein. Before we explain our results in detail, we start with a short presentation of what we will call the Lipschitz extension problem. Let \((X,\rho_X)\) be a \textit{quasi-metric space}, that is, the function \(\rho_X\colon X\times X\to \R\) is non-negative, symmetric and vanishes on the diagonal, cf. \cite[p. 827]{schoenberg1938}. Unfortunately, the term ``quasi-metric space'' has several different meanings in the mathematical literature. In the present paper, we stick to the definition given above. Let \(S\subset X\) be a subset and let \((Y, \rho_Y)\) be a quasi-metric space. A \textit{Lipschitz map} is a map \(f\colon S\to Y\) such that
the quantity
\begin{equation*}
\Lip(f):=\inf\left\{ L\geq 0  : \textrm{for all points } x,x^\prime\in S\colon \rho_Y(f(x), f(x^\prime)) \leq L \rho_X(x,x^\prime) \right\} 
\end{equation*}
is finite. We use the convention \(\inf \varnothing =+\infty\). We consider the following Lipschitz extension problem: 

\begin{question} Let  \((X,d_X)\) be a metric space, let \((Y,\rho_Y)\) be a quasi-metric space, and suppose that \(S\subset X\) is a subset of \(X\). Under what conditions on \(S, X\) and \(Y\)  is there a real number \(D\geq 1\) such that every Lipschitz map \(f\colon S\to Y\) has a Lipschitz extension \(\overline{f}\colon X \to Y\) with \(\Lip\big(\,\overline{f}\,\big)\leq D\Lip(f)\)?
\end{question}
Let \(e(X,S;Y)\) denote the infimum of the  \(D\)'s satisfying the desired property in the ``Lipschitz extension problem''. Given integers \(n, m\geq 1\), we define
\begin{equation*}
\begin{split}
e_n(X,Y)&:=\sup\big\{ e(X,S;Y) : S\subset X, \, \abs{S} \leq n\big\}, \\
e^m(X,Y)&:=\sup\big\{ e(S\cup T,S;Y) : S,T\subset X, \, S \textrm{ closed}, \, \abs{T} \leq m\big\}.
\end{split}
\end{equation*}
We use \(\abs{\,\cdot\,}\) or \(\card(\cdot)\) to denote the cardinality of a set. 
The Lipschitz extension modulus \(e_n(X,Y)\) has been studied intensively in various settings.
Nevertheless, many important questions surrounding \(e_n(X,Y)\) are still open, cf. \cite{naor2017lipschitz} for a recent overview. 

In the present article, we are interested in an upper bound for \(e^m(X,Y)\). 
We get the following result.
\begin{theorem}\label{cor:main}
Let \((X, d_X)\) be a metric space and let \((Y,\rho_Y)\) be a quasi-metric space. If \(m\geq 1\) is an integer, then 
\begin{equation}\label{main12}
e^{m}(X,Y) \leq m+1.
\end{equation}
\end{theorem}
A constructive proof of Theorem 1.1 is given in Section \ref{sec:two}. The estimate \eqref{main12} is optimal.  This follows from the following simple example. We set \(P_{m+1}:=\{0, 1, \ldots, m+1\}\subset \R\) and we consider the subset \(S=Y=\{0, m+1\}\subset P_{m+1}\) and the map \(f\colon S\to Y\) given by \(x\mapsto x\). Suppose that \(F\colon P_{m+1} \to Y\) is a Lipschitz extension of \(f\) to \(P_{m+1}\). Without effort it is verified that \(\Lip(F)=(m+1)\Lip(f)\); hence, it follows that  \eqref{main12} is sharp. The sharpness of Theorem \ref{cor:main} allows us to obtain a lower bound for the parameter \(\alpha(\omega)\) of the dichotomy theorem for metric transforms \cite[Theorem 1]{mendel2011note}, see Corollary \ref{cor:Ftrans}. If the condition that the subset \(S\subset X\) has to be closed is removed  in the definition of \(e^m(X,Y)\), then Theorem \ref{cor:main} is not valid. Indeed, if  \((X,d_X)\) is not complete and \(z\in \overline{X}\) is a point contained in the completion \(\overline{X}\) of \(X\) such that \(z\notin X\), then the identity map \(\id_X\colon X\to X\) does not extend to a Lipschitz map \(\overline{\id_X}\colon X\cup \{z\}\to X\) if we equip \(X\cup \{z\} \subset \overline{X}\) with the subspace metric. This is a well-known obstruction.  As pointed out by Naor and Mendel, there is the following upper bound of  \(e^m(X,Y)\) in terms of  \(e_m(X,Y)\) . 

\begin{lemma}[\textrm{Claim 1 in \cite{mendel2017relation}}]\label{Lem:Naor} 
Let \((X, d_X)\) and \((Y,d_Y)\) be two metric spaces. If \(m\geq 1 \) is an integer, then
\begin{equation*}
e^m(X,Y) \leq e_m(X,Y)+2.
\end{equation*}
\end{lemma} 

By the use of Lemma \ref{Lem:Naor} and \cite[Theorem 1.10]{lee2005extending}, one can deduce that if \((X,d_X)\) is a metric space and \((E, \norm{\cdot}_{_E})\) is a Banach space, then
\begin{equation*}
e^m(X,E) \lesssim \frac{\log(m)}{\log\big(\log(m)\big)}
\end{equation*}
for all integers \(m\geq 3\), where the notation \(A \lesssim B\) means  \(A \leq C B\) for some universal constant \(C\in (0, +\infty)\).
As a result, for sufficiently large integers \(m\geq 3\) the estimate in Theorem \ref{cor:main} is not optimal if we restrict the target spaces to the class of Banach spaces. In Section \ref{sec:sharpness}, we present an example that shows that for Banach space targets the estimate \eqref{main12} is sharp if \(m=1\). As a byproduct of the construction in Section \ref{sec:sharpness}, we obtain the lower bound
\begin{equation}\label{eq:estimate100}
e(\ell_2, \ell_1) \geq \sqrt{2},
\end{equation}
where \(e(\ell_2, \ell_1):=\sup\big\{ \,e(\ell_2, S; \ell_1) : S\subset \ell_2\big\}\). It is unknown if \(e(\ell_2, \ell_1)\) is finite or infinite. This question has been raised by Ball, cf. \cite{ball1992markov}. 
Let \((Y,\rho_Y)\) be a quasi-metric space and let \(F\colon [0,+\infty) \to [0,+\infty)\) be a map with \(F(0)=0\). The \textit{\(F\)-transform of \(Y\)}, denoted by \(F[Y]\), is by definition the quasi-metric space \((Y, F\circ \rho_Y)\).
Our main result can be stated as follows:
\begin{theorem}\label{hopefullySolved}
Let \((H, \langle \cdot, \cdot \rangle_{_H})\) be a Hilbert space and let \((E, \norm{\cdot}_{_E})\) be a Banach space. Suppose that \(F\colon [0,+\infty)\to [0,+\infty)\) is a map such that the composition \(F(\sqrt{\cdot})\) is a strictly-increasing concave function with \(F(0)=0\).   If \(X\subset F[H]\) is a finite subset, \(S\subset X\), and \(f\colon S\to E\) is a map, 
then there is a Lipschitz extension \(\overline{f}\colon X\to \Conv(f(S))\) such that
\begin{equation}\label{main123}
\Lip\left(\,\overline{f}\,\right) \leq  \sup_{x>0} \frac{F\big(\sqrt{m+1} \,x\big)}{F(x)}\,\Lip(f),
\end{equation}
where \(m:=\abs{X\setminus S}\).
\end{theorem}
Theorem \ref{hopefullySolved} is optimal if \(m=1\) and \(F=\id\), see Proposition \ref{prop:lower}.  Via this sharpness result we obtain that certain \(F\)-transforms of \(\ell_p\), for \(p>2\), do not isometrically embed into \(\ell_2\), see Corollary \ref{cor:snowflake}. Suppose that \(F\colon [0, +\infty)\to [0,+\infty)\) is a strictly-increasing continuous function such that the \(F\)-transform of \(\ell_2\) embeds isometrically into a Hilbert space. By a celebrated result of Schoenberg
\(F(\sqrt{\cdot})^2\) is a Bernstein function, cf. \cite[Theorem \(6^\prime\)]{schoenberg1938}; thus, the function \(F(\sqrt{\cdot})\) is concave and therefore satisfies the assumptions on \(F\) in Theorem \ref{hopefullySolved}. This provides a natural class of examples for which Theorem \ref{hopefullySolved} may be applied.
Let \(0 <\alpha \leq 1\) and \(L\geq 0\) be real numbers.
An \textit{\((\alpha, L)\)-H\"older map} is a map \(f\colon X\to Y\) such that
\[d_Y(f(x), f(x^\prime)) \leq L d_X(x,x^\prime)^\alpha\]
for all points \(x,x^\prime \in X\). 
By considering the function \(F(x)=x^\alpha\), with \(0 < \alpha \leq 1\), we obtain the following direct corollary of Theorem \ref{hopefullySolved}.

\begin{corollary}
Let \((H, \langle \cdot, \cdot \rangle_{_H})\) be a Hilbert space, let \((E, \norm{\cdot}_{_E})\) be a Banach space and let \(0 <\alpha \leq 1\) and \(L\geq 0\) be real numbers.
If \(X\subset H\) is a finite subset, \(S\subset X\), and \(f\colon S\to E\) is an \((\alpha, L)\)-H\"older map,
then there is an extension \(\overline{f}\colon X\to \Conv(f(S))\) of \(f\) such that \(\overline{f}\) is an \((\alpha, \overline{L}\,)\)-H\"older map with  
\begin{equation*}
 \overline{L} \leq \left(\sqrt{m+1}\right)^\alpha \, L,
\end{equation*}
where \(m:=\abs{X\setminus S}\).
\end{corollary}

Along the lines of the proof of Claim 1 in \cite{mendel2017relation} one can show that if \((X,d_X)\) and \((Y, d_Y)\) are metric spaces, then for all integers \(m\geq 1\) we have
\begin{equation*}
e^m(X,Y) \leq \sup_{n\geq 1} \,e^m_{n}(X,Y)+2,
\end{equation*}
where
\begin{equation*}
e_n^m(X,Y):=\sup\big\{ e(S\cup T,S;Y) : S,T\subset X,\, \abs{S} \leq n,\, \abs{T} \leq m\big\}.
\end{equation*}
Thus, by the use of Theorem \ref{hopefullySolved}, we may deduce that if \(H\) is a Hilbert space and \(E\) is a Banach space, then
\begin{equation}\label{eq:estLog}
e^m(H,E) \leq \sqrt{m+1}+2
\end{equation}
for all integers \(m\geq 1\).  In \cite[Theorem 1.12]{lee2005extending}, Lee and Naor demonstrate that \(e_n(H, E) \lesssim \sqrt{\log(n)}\) for all integers \(n\geq 2\). Thus, via this estimate (and Lemma \ref{Lem:Naor})  it is possible to obtain the upper bound
\begin{equation*}
e^m(H,E) \lesssim \sqrt{\log(m)}
\end{equation*}
that has a better asymptotic behaviour than estimate \eqref{eq:estLog}. However, since Lee and Naor use different methods, we believe that our approach has its own interesting aspects.

The paper is structured as follows. In Section \ref{sec:app}, we derive some corollaries of our main results. In Section \ref{sec:two} we prove Theorem \ref{cor:main} and in Section \ref{sec:sharpness} we show that our extension results are sharp for one point extensions. In \cite{ball1992markov}, Ball introduced the notions of Markov type and Markov cotype of Banach spaces.  To establish Theorem \ref{hopefullySolved} we estimate quantities that are of similar nature. The necessary estimates are obtained in Section \ref{sec:quadraticForm} and Section \ref{sec:betterName}. In Section \ref{sec:betterName}, we deal with M-matrices, which appear naturally in the proof of Theorem \ref{hopefullySolved}. M-matrices have first been considered by Ostrowski, cf. \cite{ostrowski1937determinanten}, and since then have been investigated in many areas of mathematics, cf. \cite{poole1974survey}. The main result of Section \ref{sec:betterName}, Theorem \ref{thm:estiMate},  may be of independent interest for the general theory of M-matrices. Finally, a proof of Theorem \ref{hopefullySolved} is given in Section \ref{sec:mainProof}. 

\section{Embeddings and indices of \(F\)-transforms}\label{sec:app}
In this section we collect some applications of our main theorems. 
Let \((X,\rho_X)\) and \((Y, \rho_Y)\) be quasi-metric spaces and let \(f\colon X\to Y\) be an injective map. We set \(\dist(f):=\Lip(f)\Lip(f^{-1})\) and
\[c_Y(X):=\inf \big\{\dist(f) : f\colon X\to Y \textrm{ injective } \big\}.\]
The sharpness of \eqref{main123} if \(m=1\) allows us to derive a necessary condition for an \(F\)-transform of an \(\ell_p\)-space to embed into a Hilbert space. 

\begin{corollary}\label{cor:snowflake} 
Let \((H, \langle \cdot, \cdot \rangle_{_H})\) be a Hilbert space and
suppose that \(F\colon [0,+\infty) \to [0,+\infty)\) is a function such that \(F(0)=0\) and
\begin{equation*}
\sup\limits_{x>0} \frac{F(x)}{x} <+\infty.
\end{equation*} 
If \(p\in [1,+\infty]\) is an extended real number and
\[\sup \big\{ c_{H}(A) : A \subset F[\ell_p], \, A \textrm{ finite } \big\} \leq 2^\epsilon, \quad \textrm{ where } \epsilon \in \big[0,\frac{1}{2}\big),\]
then \(p \leq \left(\frac{1}{2}-\epsilon\right)^{-1}\). 

\end{corollary}
The proof of Corollary \ref{cor:snowflake} is given at the end of Section \ref{sec:sharpness}. 
If \( 2 < p < +\infty\) is a real number and the \(F\)-transform \(F[\ell_p]\) embeds isometrically into a Hilbert space, then
\[F(x)=F_a(x)=
\begin{cases}
0 & x=0 \\
a & x>0 
\end{cases}
\quad \textrm{ where } a \geq 0;
\]
this follows essentially by combining a result of Kuelbs \cite[Corollary 3.1]{kuelbs1973positive} with a classical result that relates isometric embeddings to positive definite functions, cf.  for example \cite[Theorem 4.5]{WellsJamesH1975Eaei}. Furthermore, by a result of Johnson and Randrianarivony, \(\ell_p\) with \(p > 2\) does not admit a coarse embedding into \(\ell_2\), cf. \cite{10.2307/4098068,mendel2008metric}. 

We proceed with an application of Theorem \ref{cor:main}. 
Let \(F\colon [0,+\infty)\to [0,+\infty)\) be a function with \(F(0)=0\).
Suppose that \(F\) is subadditive and strictly increasing.  We define
\[D_F(\alpha)=\sup_{x >0} \frac{F(\alpha x)}{F(x)}\]
for all \(\alpha \geq 0\). Clearly, the function \(D_F\colon [0,+\infty)\to [0,+\infty) \) is finite, submutliplicative and non-decreasing. Moreover, 
\[F(\alpha x) \leq D_F(\alpha) F(x)\]
for all real numbers \(x, \alpha \geq 0\). The \textit{upper index} of \(F\) is defined by
\begin{equation}\label{eq:limit1}
\beta(F)=\lim_{\alpha \to +\infty} \frac{\log(D_F(\alpha))}{\log(\alpha)}.
\end{equation}
The existence of the limit \eqref{eq:limit1} may be deduced via the general theory of subadditive functions, since \(D_F\) is submultiplicative and non-decreasing, cf. \cite[Remark 1.3 (b)]{LechMaligranda1985}. We have \(0 \leq \beta(F) \leq 1\), for \(F\) is subadditive. If \((X,d_X)\) is a metric space, we set
\[c_F(X):=\inf\big\{ c_{F[Y]}(X) : (Y, d_Y) \textrm{ metric space } \big\}.\]
In \cite[Theorem 1]{mendel2011note}, Mendel and Naor obtained a dichotomy theorem for the quantity \(c_F(X)\), if \(F\) is concave and non-decreasing. The upper index of \(F\) allows us to obtain lower bounds for the rate of growth of \(c_F(P_n)\), where \(P_n:=\{0,1, \ldots, n\}\subset \R\).

\begin{corollary}\label{cor:Ftrans}
Let \(F\colon [0,+\infty)\to [0,+\infty)\) be a strictly-increasing subadditive function with \(F(0)=0\). 
If \(\,0\leq \alpha < 1-\beta(F)\,\) is a real number, then there exists an integer \(N\geq 1\) such that
\[n^\alpha \leq c_F(P_n)\]
for all \(n \geq N\). 
\end{corollary}
\begin{proof}
We may assume that \(\beta(F)<1\). 
Let \((Y, \rho_Y)\) be a quasi-metric space and let \((X, d_X)\) be a metric space. We may employ Theorem \ref{cor:main} to conclude that
\begin{equation}\label{eq:Ftrans}
e^m\big(F[X], Y\big) \leq \sup_{x >0} \frac{F\big((m+1) x\big)}{F(x)},
\end{equation}
for all integers \(m\geq 0\). We set \(Y_{m}:=\{0, m\}\subset P_{m}\). Since
\[e^{m-1} \big(P_{m}, Y_{m}\big)=m,\]
inequality \eqref{eq:Ftrans} asserts that 
\begin{equation}\label{eq:lowerBoundF}
m\leq \sup_{x >0} \frac{F( m x)}{F(x)} c_F(P_{m})=D_F(m) c_F(P_{m})
\end{equation}
for all \(m\geq 1\). 
Let \(\epsilon >0\) be a real number such that \(\alpha< 1-\beta(F)-\epsilon\). By the virtue of Theorem 1.2 in \cite{LechMaligranda1985} there exists a real number \(C\geq 0\) such that
\[D_F(\alpha) \leq \alpha^{\beta(F)+\epsilon}\]
for all \(\alpha \geq C\). Consequently, by the use of \eqref{eq:lowerBoundF} we obtain for all \(n\geq N:= \lceil C\rceil \) that
\[n^\alpha \leq n^{1-\beta(F)-\epsilon}\leq c_F(P_n),\]
as desired. 
\end{proof}
As a consequence of Corollary \ref{cor:Ftrans}, we conclude that if \(\beta(F)<1\), then the second possibility of the dichotomy \cite[Theorem 1]{mendel2011note} holds. 
Thus, there is the following natural question: If \(\beta(F)=1\), is it true that, then \(c_F(X)=1\) for all finite metric spaces \((X, d_X)\)?



\section{Proof of Theorem \ref{cor:main}}\label{sec:two}
In this section, we derive Theorem \ref{cor:main}. 

\begin{proof}[Proof of Theorem \ref{cor:main}]
Let \(S\subset X\) be a closed subset and let \(T\subset X\) be a finite subset such that \(S\cap T=\varnothing \) and \(\abs{T}\leq m\). Let \(f\colon S\to Y\) be a Lipschitz map. In what follows we construct for each \(\epsilon >0\) a map \(F_\epsilon\colon S\cup T\to Y\) that is a Lipschitz extension of \(f\) to \(S\cup T\) such that
 \(\Lip(F_{\epsilon})\leq \left((1+\epsilon)m+1\right)\Lip(f)\). We start with a few definitions. Fix \(\epsilon >0\). Let \(F\subset S\) be a finite subset such that for each point \(z\in T\) there is a point \(x\in F\) with
\begin{equation}\label{eq:inequality}
d_X(z,x) \leq (1+\epsilon)d_X(z, S).
\end{equation}
Since \(S\) is closed and \(T\) is finite, such a set \(F\) clearly exists. 
We set
\begin{equation*}
E:=\big\{ \{u,v\} : u\neq v \textrm{ with } \left(u,v\in T\right) \textrm{ or } \left(u\in T, v\in F \right) \big\}.
\end{equation*}
Let \(G:=(V,E)\) denote the graph with vertex set \(V:=F\cup T\) and edge set \(E\). We say that a subset  \(E^\prime\subset E\) is \textit{admissible} if the graph \(G^\prime:=(V,E^\prime)\) contains no cycles
and has the property that if  \(v,v^\prime \in F\) are distinct, then there is no path in \(G^\prime\) connecting them. For each edge \(\{u,v\}\in E\) we set \(\omega(\{u,v\}):=d_X(u,v)\). Furthermore, let \(N\geq 0\) denote the cardinality of \(E\). 
Let \(e\colon \{1, \ldots, N\}\to E\) be a bijective map such that the composition \(\omega\circ e\) is a non-decreasing function. We construct the sequence \(\{E_\ell\}_{\ell=0}^{N}\) of subsets of \(E\) via the following recursive rule:
\begin{equation}\label{Eq:ConstructionP}
E_0\coloneqq\varnothing, \quad
E_\ell\coloneqq
\begin{cases}
\{e(\ell)\}\cup E_{\ell-1} & \textrm{ if  } \{e(\ell)\}\cup E_{\ell-1} \textrm{  is admissible} \\
E_{\ell-1} & \textrm{ otherwise}.
\end{cases}
\end{equation}
We claim that for each point \(z\in T\) there exists an integer \(L_z\geq 1\) and a unique injective path \(\gamma_z\colon \{1, \ldots, L_z\}\to E_N\) connecting \(z\) to a point \(x_z\) in \(F\). Indeed, the uniqueness part of the claim follows directly, as \(E_N\) is admissible.
Now, we show the existence part. Let \(z\in T\) be a point. Choose an arbitrary point \(x\in F\). If the edge \(\{x,z\}\) is contained in \(E_N\), then an injective path \(\gamma_z\) with the desired property surely exists. Suppose now that \(\{x,z\}\notin E_N\). It follows from the recursive construction of \(E_N\) that in this case there either exists a path in \(E_N\) from \(z\) to \(x\) of length greater than or equal to two or there exists a path in \(E_N\) from \(z\) to a point \(x^\prime\in F\) distinct from \(x\). Thus, in any case an injective path \(\gamma_z\) with the desired properties exists. We define the map \(F_\epsilon\colon S\cup T \to Y\) as follows
\begin{equation*}
\begin{split}
&F_\epsilon(x):=f(x) \quad\quad\quad\,\,\,\,\,\textrm{ for all } x\in S  \\
&F_\epsilon (z):=f(x_z) \quad\quad\quad\,\,\textrm{ for all } z\in T. 
\end{split}
\end{equation*}
In other words, \(F_\epsilon=f\circ R_\epsilon\), where \(R_\epsilon\colon S\cup T\to S\) is the retraction that maps \(z\in T\) to \(x_z\in S\). 
In what follows, we show that \(R_\epsilon\) has Lipschitz constant smaller than or equal to  \((1+\epsilon)m+1\). This is the reason that enables us to put so low requirements onto ‘distance’ in Y. Now, let \(z\in T\) and \(x\in S\) be points. By the use of the triangle inequality, we compute
\begin{equation}\label{eq:sillyestimate}
\begin{split}
&\rho_Y(F_{\epsilon}(x), F_{\epsilon}(z))=\rho_Y(f(x), f(x_z))\leq \Lip(f)d_X(x,x_z) \\
&\leq \Lip(f)\left(d_X(x,z)+ \sum_{\ell=1}^{L_z} \omega(\gamma_z(\ell))\right).\\
\end{split}
\end{equation}
Let \(x^\prime \in F\) be a point such that the pair  \((z, x^\prime)\) satisfies the estimate \eqref{eq:inequality}. By the recursive construction of \(E_N\), it follows that \(\omega(\gamma_z(\ell)) \leq d(x^\prime, z)\) for all \(\ell \in \{1, \ldots, L_z\}\), since the function \(\omega\circ e\) is non-decreasing. Hence, by the use of \eqref{eq:sillyestimate} we obtain
\begin{equation*}
\begin{split}
&\rho_Y(F_{\epsilon}(x), F_{\epsilon}(z)) \\
&\leq  \Lip(f)\left(d_X(x, z)+ L_z d_X(x^\prime, z)\right)\\
&\leq \Lip(f)\left(1+L_z(1+\epsilon)\right)d_X(x, z)\\
&\leq \Lip(f)\left( (1+\epsilon)m+1\right)d_X(x, z).
\end{split}
\end{equation*}
Now, let \(z,z^\prime \in T\) be points. If \(x_z=x_{z^\prime}\), then \(F_\epsilon(z)=F_\epsilon(z^\prime)\), by construction. 
Suppose now that \(x_z\neq x_{z^\prime}\) . We compute
\begin{equation}\label{eq:rushh}
\begin{split}
&\rho_Y(F_{\epsilon}(z), F_{\epsilon}(z^\prime))=\rho_Y(f(x_z), f(x_{z^\prime}))\leq \Lip(f)d_X(x_z,x_{z^\prime}) \\
&\leq \Lip(f)\left(\sum_{\ell=1}^{L_z} \omega(\gamma_z(\ell)) +d_X(z, z^\prime)+\sum_{\ell=1}^{L_{z^\prime}} \omega(\gamma_{z^\prime}(\ell))\right). 
\end{split}
\end{equation}
The edge \(\{z, z^\prime\}\) is not contained in \(E_N\); thus, by the recursive construction of \(E_N\) we obtain that \(\omega(\gamma_z(\ell)) \leq \omega(\{z, z^\prime\})\) for all \(\ell\in \{1, \ldots, L_z\}\) and \(\omega(\gamma_{z^\prime}(\ell)) \leq \omega(\{z, z^\prime\})\) for all for all \(\ell\in \{1, \ldots, L_{z^\prime}\}\). By virtue of \eqref{eq:rushh} we deduce
\begin{equation*}
\begin{split}
&\rho_Y(F_{\epsilon}(z), F_{\epsilon}(z^\prime))\\
&\leq \Lip(f)\left( L_z+1+L_{z^\prime} \right) d_X(z,z^\prime) \\
&\leq \Lip(f)(m+1)d_X(z,z^\prime). 
\end{split}
\end{equation*}
The last inequality follows, since \(E_N\) is admissible and the paths \(\gamma_z, \gamma_{z^\prime}\) are injective; thus, \(L_z+L_{z^\prime} \leq m\).
We have considered all possible cases and we have established that
\begin{equation*}
\Lip(F_{\epsilon})\leq \left((1+\epsilon)m+1\right)\Lip(f),
\end{equation*}
as desired. This completes the proof.  
\end{proof}


\section{One point extensions of Banach space valued maps}\label{sec:sharpness}
The collection of examples that we construct in this section is inspired by \cite{grunbaum1960projection}. We define the sequence \(\{W_k\}_{k\geq 0}\) of matrices via the recursive rule
\begin{equation*}
\begin{split}
&W_0:=1, \\
&W_{k+1}:=\begin{pmatrix}
W_k & W_k \\
W_k & -W_k
\end{pmatrix}. \\
\end{split}
\end{equation*}
The matrices \(W_k\) are commonly known as \textit{Walsh matrices}. 
For each integer \(k\geq 1\) let \(W_k^{\hspace{0.08em}\prime}\) denote the \((2^k-1)\times 2^k\) matrix that is obtained from \(W_k\) by deleting the first row of \(W_k\). Further, for each integer \(k\geq 1\) and each integer \(\ell \in \{1, \ldots, 2^k\}\)  we set
\begin{equation}\label{eq:vertices}
v_\ell^{(k)}\coloneqq \ell\textrm{-th column of the matrix } W_k^{\hspace{0.08em}\prime}.  
\end{equation} 
By construction, \(v_\ell^{(k)}\in \R^{2^k-1}\) for all \(k\geq 1\) and \(\ell \in \{1, \ldots, 2^k\}\). Clearly, \(v_\ell^{(k)}\in \ell_{p}\) for all \(p\in [1, +\infty]\) via the canonical embedding. The goal of this section is to prove the following proposition.

\begin{proposition}\label{prop:lower}
Let \(p\in[1, +\infty]\) be an element of the extended real numbers and let \(k\geq 1\) be an integer. If \(F\colon \left(\{v_1^{(k)} , \ldots , v_{2^k}^{(k)}\}\cup\{0\},\norm{\cdot}_p\right)\to \left(\ell_1, \norm{\cdot}_1\right)\) is a Lipschitz extension of the function 
\begin{equation*}
\begin{split}
&f\colon \left(\{v_1^{(k)} , \ldots , v_{2^k}^{(k)}\}, \norm{\cdot}_p\right) \to \left(\ell_1, \norm{\cdot}_1\right)\\
& v_\ell^{(k)}\mapsto v_\ell^{(k)},
\end{split}
\end{equation*}
then it holds that
\begin{equation*}
\Lip(F) \geq \left(2-\frac{1}{2^{k-1}}\right)^{\frac{1}{p_\star}} \Lip(f) ,
\end{equation*}
where \(1/p_\star:=1-1/p\) if \(p\neq +\infty\) and \(1/p_\star:=1\) otherwise. 
\end{proposition}
Note that Proposition \ref{prop:lower} implies in particular that \(e(\ell_2, \ell_1) \geq \sqrt{2}\). 
The key component in the proof of Proposition \ref{prop:lower} is the following geometric lemma. 

\begin{lemma}\label{Lem:Lower}

Let \(k\geq 1\) be an integer and suppose that \(w\in \R^{2^k-1}\) is a vector such that
\begin{equation}\label{Eq:Inequality}
\norm{v_\ell^{(k)}-w}_1\leq \norm{v_\ell^{(k)}}_1 \textrm{ for all } \ell\in \{ 1, \ldots, 2^k\},
\end{equation}
then it holds that \(w=0\). 
\end{lemma}
\begin{proof}
By the use of a simple induction it is straightforward to show that 
\begin{equation}\label{eq:sumToZero}
\sum_{\ell=1}^{2^k} v_\ell^{(k)}=0.
\end{equation}
Moreover, since \(v_\ell^{(k)}\) is a \(\pm 1\) vector, inequality \eqref{Eq:Inequality} implies that
\[\langle w, v_\ell^{(k)}\rangle_{_{\R^{2^k-1}}}\leq 0.\]
Equality \eqref{eq:sumToZero} implies that none of these inequalities can be strict;
thus, as (for instance) the vectors \(v_2^{(k)}, \ldots, v_{2^k}^{(k)}\) form a basis of \(\R^{2^k-1}\), we obtain \(w=0\), as desired. 
\end{proof} 
Having Lemma \ref{Lem:Lower} at our disposal, Proposition \ref{prop:lower} can readily be verified.
\begin{proof}[Proof of Proposition \ref{prop:lower}]
To begin, we compute \(\Lip(f)\). We claim that 
\begin{equation}\label{eq:Lip}
\Lip(f)=\left(2^{k-1}\right)^{\frac{1}{p_\star}}.
\end{equation} 
First, suppose that \(p\in [1, +\infty)\). A simple induction implies that two distinct columns of \(W_k\) are orthogonal to each other. Since the entries of \(W_k\) consist only of plus and minus one, we obtain that
\begin{equation*}
\norm{ v_i^{(k)}-v_j^{(k)}}_p^p=2^p\card\left(\left\{ \ell\in\{1, \ldots, 2^{k}-1\} : (v_i^{(k)})_\ell\neq (v_j^{(k)})_\ell \right\}\right)=2^p2^{k-1},
\end{equation*}
where we use \(\card(\cdot)\) to denote the cardinality of a set.
Hence, if \(p\in [1, +\infty)\), then the identity \eqref{eq:Lip} follows. Since the \(p\)-norms \(\norm{\cdot}_p\) converge pointwise to the maximum norm \(\norm{\cdot}_\infty\) if \(p\to +\infty\), the identity \eqref{eq:Lip} follows also in the case \(p=+\infty\), as was left to show. By considering the contraposition of the statement in Lemma \ref{Lem:Lower}, we may deduce that there is an index \(\ell\in \{1, \ldots, 2^k\}\) such that 
\begin{equation*}
\norm{v_\ell^{(k)}-F(0)}_1 \geq \norm{v_\ell^{(k)}}_1.
\end{equation*}
As a result, we obtain that
\begin{equation*}
\Lip(F)\geq \frac{\norm{v_\ell^{(k)}-F(0)}_1}{\norm{v_\ell^{(k)}}_p}\geq \frac{\norm{v_\ell^{(k)}}_1}{\norm{v_\ell^{(k)}}_p}=(2^k-1)^{\frac{1}{p_\star}}.
\end{equation*}
Hence, it follows that
\begin{equation*}
\frac{\Lip(F)}{\Lip(f)}\geq \frac{(2^k-1)^{\frac{1}{p_\star}}}{(2^{k-1})^{\frac{1}{p_\star}}}=\left(2-\frac{1}{2^{k-1}}\right)^{\frac{1}{p_\star}};
\end{equation*}
as desired. 
\end{proof}

We conclude this section with the proof of Corollary \ref{cor:snowflake}. 

\begin{proof}[Proof of Corollary \ref{cor:snowflake}]
Let \(k\geq 1\) be an integer and let \[g_F\colon \left(\{v_1^{(k)} , \ldots , v_{2^k}^{(k)}\}, F\circ \norm{\cdot}_p\right) \to \left(\ell_1, \norm{\cdot}_1\right)\] denote the map such that \(v_i^{(k)}\mapsto v_i^{(k)}\). The vectors \(v_i^{(k)}\) are given as in \eqref{eq:vertices} and interpreted as elements of \(\ell_p\) via the canonical embedding. It is readily verified that 
\begin{equation*}
\Lip \left( g_{F} \right)= \frac{A}{F(A)} \Lip \left( g_{\id}\right),
\end{equation*}
where \(A:=\norm{ v_i^{(k)}-v_j^{(k)}}_p\). 
Now, let \(\delta >0\)  be a real number. Using the assumptions in Corollary \ref{cor:snowflake} and Theorem \ref{hopefullySolved} (for the map \(F=\id\)) it follows that there is a map
\(G_F\colon \left(\{v_1^{(k)} , \ldots , v_{2^k}^{(k)}\}\cup\{0\},F\circ\norm{\cdot}_p \right)\to \left(\ell_1, \norm{\cdot}_1\right)\) that extends \(g_F\)
such that 
\begin{equation*}
\Lip\left(G_F\right) \leq (1+\delta)\, 2^\epsilon \,\sqrt{2} \,\Lip\left(g_F\right).
\end{equation*}
We define the map \(T\colon  \left(\{v_1^{(k)} , \ldots , v_{2^k}^{(k)}\}\cup\{0\},\norm{\cdot}_p \right)\to \left(\ell_1, \norm{\cdot}_1\right)\) via \(x\mapsto G_F(x)\). 
We calculate
\begin{equation*}
\Lip(T)\leq (1+\delta)\, 2^\epsilon \,\sqrt{2} \, \max\left\{\frac{F(A)}{A}, \frac{F(B)}{B}\right\} \Lip\left(g_F\right),
\end{equation*}
where \(B:=\norm{ v_i^{(k)}-0}_p\).
Since the map \(T\) is a Lipschitz extension of \(g_{\id}\), Proposition \ref{prop:lower} tells us that
\begin{equation*}
\Lip(T) \geq \left(2-\frac{1}{2^{k-1}}\right)^{\frac{1}{q}} \Lip\left(g_{\id}\right)=\frac{A}{B}  \left(1-\frac{1}{2^{k}}\right)\Lip\left(g_{\id}\right),
\end{equation*}
where  \(1/q:=1-1/p\) if \(p\neq +\infty\) and \(1/q:=1\) otherwise. 
We set \(\gamma:=\frac{A}{B}\). Thus, by putting everything together and via a simple scaling argument, we obtain for all \(x >0\)
\begin{equation*}
\gamma  \left(1-\frac{1}{2^{k}}\right) \frac{ F(\gamma x) }{\gamma x} \leq (1+\delta) 2^\epsilon \sqrt{2}\, \max\left\{\frac{F(x)}{x},\frac{F(\gamma x)}{\gamma x}\right\}.
\end{equation*}
 Thus, since 
\[\sup\limits_{x>0} \frac{F(x)}{x} <+\infty \]
we obtain 
\[\frac{\sqrt[q]{2}}{\sqrt[p]{1-\frac{1}{2^k}}} \left(1-\frac{1}{2^{k}}\right)=\gamma  \left(1-\frac{1}{2^{k}}\right) \leq  (1+\delta) 2^\epsilon \sqrt{2}.\]
Consequently, as \(k\geq 1\) and \(\delta >0\) are arbitrary, we deduce \(p \leq \left( \frac{1}{2}-\epsilon \right)^{-1}\). This completes the proof.
\end{proof}

\section{Minimum value of a certain quadratic form in Hilbert space}\label{sec:quadraticForm}
Let \((H, \langle \cdot, \cdot \rangle_{_H})\) be a Hilbert space, let \(I\) denote a finite set and let \(\mathbf{x}\colon I \to H\) be a map. Suppose that \(\boldsymbol\lambda\colon I \times I \to \R\) is a symmetric, non-negative function. Further, assume that \(G\colon [0,+\infty)\to [0,+\infty)\) is a convex, non-decreasing function with \(G(0)=0\). 
We define
\[\Phi(\mathbf{x}, \boldsymbol{\lambda}, G):=\sum_{(k, \ell) \in I\times I} \boldsymbol\lambda\big(k, \ell\big) \,G\big(\norm{\mathbf{x}(k)-\mathbf{x}(\ell)}_{_H}^2\big)\]
and for each subset \(J\subset I\) we set
\[\mathsf{m}(\mathbf{x}, \boldsymbol{\lambda}, G, J):=\inf\big\{\Phi(\mathbf{z}, \boldsymbol{\lambda}, G) \,:\, \mathbf{z}\colon I\to H \textrm{ is a map with } \mathbf{z}|_{J^c}=\mathbf{x}|_{J^c} \big\}.\]
The remainder of this section is devoted to calculate the quantity \(\mathsf{m}(\mathbf{x}, \boldsymbol{\lambda}, \id, J)\). Let \(J\subset  I\) be a proper subset. We may suppose that \(J=\big\{1, \ldots, m\big\}\), where \(m:=\card(J)\). To ease notation, we set \(\lambda_{k\ell}:=\bm{\lambda}(k,\ell)\) and  we define the matrix 
\begin{equation}\label{eq:MMatrix}
 M(\bm{\lambda}, J):=
\begin{bmatrix}
  \sum\limits_{k\in J^c} \lambda_{1k}+\sum\limits_{j=1}^m \lambda_{1j}     & -\lambda_{12} & \dots & -\lambda_{1m} \\
    -\lambda_{21}       & \sum\limits_{k\in J^c} \lambda_{2k}+\sum\limits_{j=1}^m \lambda_{2j} &  \dots & -\lambda_{2m} \\
    \vdots & \vdots & \ddots & \vdots \\
    -\lambda_{m1}       & -\lambda_{m2} &  \dots & \sum\limits_{k\in J^c} \lambda_{mk}+\sum\limits_{j=1}^m \lambda_{mj}
\end{bmatrix}.
\end{equation}
The matrices \(M(\bm{\lambda}, J)\) appear naturally in the proof of Theorem \ref{hopefullySolved}. If the symmetric matrix \(M:=M(\bm{\lambda}, J)\) is strictly diagonally dominant, that is, for each integer \(1 \leq i \leq m\), it holds
\begin{equation*}
\abs{m_{ii}} > \sum_{j\neq i}^m \abs{m_{ij}},
\end{equation*}
it follows via Gershgorin's circle theorem that \(M\) is positive definite.
As a result, the matrix \(M(\bm{\lambda}, J)\) is non-singular if  \[\sum\limits_{k\in J^c} \lambda_{ik} >0 \quad \textrm{ for all } 1 \leq i \leq m.\] 
Next, we deduce the minimum value of \(\mathsf{m}(\mathbf{x}, \boldsymbol{\lambda}, \id, J)\). 
\begin{proposition}\label{prop:quadratic}
Let \((H, \langle \cdot, \cdot \rangle_{_H})\) be a Hilbert space, let \(I\) be a finite set and let \(\mathbf{x}\colon I \to H\) be a map. Suppose that \(\boldsymbol\lambda\colon I \times I \to \R\) is a symmetric, non-negative function and let \(J\subset I\) be a proper subset. If the matrix \(M:=M(\bm{\lambda}, J)\) given by \eqref{eq:MMatrix} is strictly diagonally dominant and \(\lambda_{k\ell}=0\) for all  \(k, \ell \in J^c\), then 
\begin{equation}\label{eq:weekend}
\begin{split}
\mathsf{m}(\mathbf{x}, \boldsymbol{\lambda}, \id, J)=\sum_{i\in J }\sum_{j\in J} \sum_{k\in J^c} \sum_{\ell\in J^c } \lambda_{ik} c_{ij} \lambda_{j\ell} \norm{\mathbf{x}(k)-\mathbf{x}(\ell)}^{2}_{_H} \\
\end{split}
\end{equation}
where \(C:=M^{-1}\). Moreover,
\begin{equation}\label{eq:SumToOne}
\sum_{j=1}^{\abs{J}} c_{ij}\sum\limits_{k\in J^c} \lambda_{jk}=1
\end{equation}
for all integers \(1 \leq i \leq \abs{J}\). 
\end{proposition}
\begin{proof}
We set \(m:=\abs{J}\). We may suppose that \(J=\{1, \ldots, m\}\). Since \(D^{-1}M\bm{j}=\bm{j}\), where \(\bm{j}:=(1, \ldots, 1)\in \R^m\) and \(D:=(d_{ij})_{1 \leq i, j\leq m}\) is a diagonal matrix with
\[d_{ii}:=\sum\limits_{k\in J^c} \lambda_{ik}, \quad \textrm{ for all } 1 \leq i \leq m,\] we obtain
\(CD\bm{j}=\bm{j}\), that is,
\begin{equation}\label{eq:sumofOne}
\sum_{j=1}^m c_{ij}\sum\limits_{k\in J^c} \lambda_{jk}=1
\end{equation} 
for all \(1 \leq i \leq m\). Thus, \eqref{eq:SumToOne} follows. 
Let the map \(\Phi \colon H^m\to \R\) be given by the assignment 
\begin{equation*}
(z_1, \ldots, z_m)\mapsto  \sum_{i=1}^m\sum\limits_{k\in J^c}\lambda_{ik}\norm{z_i-\mathbf{x}(k)}_{_H}^2+\frac{1}{2}\sum_{i=1}^m\sum_{j=1}^m \lambda_{ij} \norm{z_i-z_j}^{2}_{_H}.
\end{equation*}
Note that \[ 2 \inf \Phi=\mathsf{m}(\mathbf{x}, \boldsymbol{\lambda}, G, J).\] Thus, to conclude the proof we calculate the minimum value of the map \(\Phi\).
 Let \(U\subset H\) denote the span of the vectors \(\big(\mathbf{x}(k)\big)_{k\in J^c}\). Clearly, \(\inf \Phi|_U=\inf \Phi\). 
In the following, we compute the minimal value of \(\Phi|_U\). 

 The subset \(U\subset H\) is linearly isometric to \((\R^d, \norm{\cdot}_{_2})\) for some integer \(1 \leq d \leq \card(J^c)\). Consequently, we may
suppose (by abuse of notation) for all \(k\in J^c\) that \(\mathbf{x}(k)\in \R^d\), say \(\mathbf{x}(k)=(x_{k1}, \ldots, x_{kd})\), and that the function \(\Phi|_U\colon (\R^d)^m\to \R \) is given by the assignment
\begin{equation*}
(p_1, \ldots, p_m)\mapsto \sum_{t=1}^d \left(\sum_{i=1}^m\sum_{j=1}^m p_{it}m_{ij} p_{jt}-2\sum_{i=1}^m p_{it}\sum\limits_{k\in J^c} \lambda_{ik}x_{rk}+\sum_{i=1}^m\sum\limits_{k\in J^c} \lambda_{ik}x_{kt}^2\right),
\end{equation*}
where \(p_i:=(p_{i1}, \ldots, p_{id})\) for all integers \(1 \leq i \leq m\). 
Using elementary analysis, one can deduce that the minimum value of \(\Phi|_U\) is equal to
\begin{equation}\label{eq:SophieHunger}
\sum_{t=1}^d \left(-\sum_{i=1}^m \sum_{j=1}^m \sum_{r=1}^n \sum_{s=1}^n  \lambda_{js}c_{ij}\lambda_{ir}x_{st}x_{rt}+\sum_{i=1}^m\sum_{r=1}^n \lambda_{ir}x_{rt}^2\right).
\end{equation}
Thus, via \eqref{eq:SophieHunger} and \eqref{eq:sumofOne} we conclude that the minimum value of \(\Phi\) is equal to
\begin{equation*}
\begin{split}
&\sum_{t=1}^d \left(\sum_{i=1}^m \sum_{j=1}^m \sum\limits_{k\in J^c} \sum\limits_{\ell\in J^c}  \lambda_{j\ell}c_{ij}\lambda_{ik}\left(-x_{\ell t}x_{kt}+x_{kt}^2\right)\right)\\
&=\frac{1}{2} \sum_{i=1}^m \sum_{j=1}^m \sum\limits_{k\in J^c} \sum_{\ell\in J^c}  \lambda_{j\ell}c_{ij}\lambda_{ik}\left(\sum_{t=1}^d (x_{\ell t}-x_{kt})^2 \right)\\
& =\frac{1}{2} \sum_{i=1}^m \sum_{j=1}^m \sum\limits_{k\in J^c} \sum\limits_{\ell\in J^c}  \lambda_{j\ell}c_{ij}\lambda_{ik} \norm{\mathbf{x}(\ell)-\mathbf{x}(k)}^{2}_{_H}, 
\end{split}
\end{equation*}
as claimed. This completes the proof. 
\end{proof}


\section{An inequality involving the entries of an M-matrix and its inverse}\label{sec:betterName}
A matrix \(M\in \Mat(m\times m; \R)\) with non-positive off-diagonal elements is said to be an \textit{M-matrix} if \(M\) is non-singular and each entry of \(M^{-1}\) is non-negative, cf. \cite[Definition 1.1]{markham1972nonnegative}. There are several equivalent definitions of an M-matrix, cf. \cite{fiedler1962matrices}.  M-matrices and their matrix inverses are generally well understood, cf. \cite{poole1974survey, johnson1982inverse} for a survey of the theory. 
A primary example of M-matrices are matrices \(M:=M(\bm{\lambda}, J)\). Indeed, such matrices are strictly diagonally dominant (thus non-singular) and via Gauss elimination it is straightforward to show that each entry of the inverse of \(M(\bm{\lambda}, J)\) is non-negative.It is worth to point out that a matrix \(M\in \Mat(m\times m;\R)\) with non-positive off-diagonal elements is an M-matrix if and only if there are matrices \(W, D \in \Mat(m\times m;\R)\) such that \(W\) is a strictly diagonally dominant M-matrix, \(D\) is a diagonal matrix with positive diagonal elements and \(M=WD\). This is a classical result of Fiedler and Pták, cf. \cite[Theorem 4.3]{fiedler1962matrices}. The following result will play a major role in the proof of Theorem \ref{hopefullySolved}. 

\begin{theorem}\label{thm:estiMate}
Let \(m\geq 2\) and let \(M\in \Mat(m\times m;\R)\) be a symmetric invertible matrix with non-positive off-diagonal elements. We set \(C:=M^{-1}\). If \(M\) is an M-matrix, then 
\begin{equation}\label{eq:estimateWewant}
\frac{1}{2}\sum_{i=1}^m \sum_{j=1}^m \abs{m_{ij}}\abs{c_{ik}c_{j\ell}-c_{jk}c_{i\ell}} \leq (m-1)c_{k\ell}
\end{equation}
for all integers \(1 \leq k,\ell \leq m \) with \(k\neq \ell\).
\end{theorem}

The estimate in Theorem \ref{thm:estiMate} is sharp. This is the content of the following example.
\begin{example}
Let \(m\geq 2\) be an integer and let \(M\in \Mat(m\times m;\R)\) be the tridiagonal matrix given by
\begin{equation*}
m_{ij}:=
\begin{cases}
3 & \textrm{if } i=j \\
-1& \textrm{if } i=j-1 \\
-1& \textrm{if } i=j+1 \\
0 & \textrm{otherwise}.
\end{cases}
\end{equation*}
Clearly, \(M\) is a symmetric M-matrix. As usual, we set \(C:=M^{-1}\). Since \(\det{(M)} C=\textrm{\normalfont adj}(M)\), where \(\textrm{\normalfont adj}(M)\) is the adjugate matrix of \(M\), it follows
\begin{equation}\label{eq:firsteq}
c_{1m}=\frac{1}{\det{M}}.
\end{equation}
Furthermore, via Jacobi's equality \cite{Jacobi1841}, see \eqref{eq:Jacobi}, we get
\begin{equation}\label{eq:secondeq}
\abs{c_{i1}c_{jm}-c_{j1}c_{im}}=\frac{\abs{\det{M\big[[m]\setminus{\{1,m\}},[m]\setminus{\{i,i+1\}}\big]}}}{\det{M}}=\frac{1}{\det{M}}
\end{equation}
for all pairs of integers \((i,j)\) with \(i=j-1\). By virtue of \eqref{eq:firsteq} and \eqref{eq:secondeq} we obtain
\begin{equation*}
\frac{1}{2}\sum_{i=1}^m \sum_{j=1}^m \abs{m_{ij}\left(c_{i1}c_{jm}-c_{j1}c_{im}\right)}=\frac{m-1}{\det{M}}=(m-1)c_{1m}. 
\end{equation*} 
Consequently, the estimate \eqref{eq:estimateWewant} is best possible. 
\end{example}

This section is structured as follows. To begin, we gather some information 
that is needed to prove Theorem \ref{thm:estiMate}. 
At the end of the section, we establish Theorem \ref{thm:estiMate}.

We start with a lemma that calculates the sum in \eqref{eq:estimateWewant} if the absolute values from the \(2\times 2\)-minors are removed. 
\begin{lemma}\label{lem:zeroRowSum}
Let \(m\geq 2\) and let \(M\in \Mat(m\times m; \R)\) be an M-matrix. We set \(C:=M^{-1}\). 
If \(1 \leq k, \ell \leq m\) are distinct integers, then
\begin{equation}\label{eq:firstequation}
\begin{split}
\sum_{j=1}^m \abs{m_{kj}}(c_{kk}c_{j\ell}-c_{jk}c_{k\ell})=c_{k\ell},
\end{split}
\end{equation}
and for all  integers \(1 \leq i \leq m\) with \(i\neq k,\ell\), 
\begin{equation}\label{eq:secondequation}
\sum_{j=1}^m \abs{m_{ij}}(c_{ik}c_{j\ell}-c_{jk}c_{i\ell})=0. 
\end{equation}
\end{lemma}
\begin{proof}
Since \(C\) is the matrix inverse of \(M\), we compute
\begin{equation*}
\begin{split}
&\sum_{j=1}^m m_{ij}c_{ik}c_{j \ell}=\delta_{i \ell}c_{ik}, \\
&\sum_{j=1}^m m_{ij}c_{jk}c_{i\ell}=\delta_{ik}c_{i\ell}
\end{split}
\end{equation*}
for all \(1 \leq i \leq m\).
As a result, we obtain
\begin{equation*}
\sum_{j=1}^m m_{ij}(c_{ik}c_{j \ell}-c_{jk}c_{i\ell})=\delta_{i \ell}c_{ik}-\delta_{ik}c_{i\ell}.
\end{equation*}
Therefore, the desired equalities follow, since \(m_{ij} \leq 0\) for all distinct integers \(1 \leq i, j \leq m\). 
\end{proof}

We proceed with the following corollary. 
\begin{corollary}[zero pattern of inverse M-matrices]\label{Cor:M-mat}
Let \(m\geq 2\) and let \(M\in \Mat(m\times m;\R)\) be an invertible matrix with non-positive off-diagonal elements. We set \(C:=M^{-1}\).
If \(M\) is an M-matrix and \(k, \ell\in \{1, \ldots, m\}\) are two distinct integers such that \(c_{k \ell}=0\), then
\begin{enumerate}[label=(\roman*)]
\item\label{it:1} for all integers \(i\in \{1, \ldots, m\}\), \(m_{ki}=0\) or \(c_{i\ell}=0\). In particular, \(m_{k \ell}=0\).
\item\label{it:2} for all integers \(i\in \{1, \ldots, m\}\), \(m_{ki}=0\) or \(m_{i \ell}=0\).
\item\label{it:3} the matrix \(M\) has at least \(m-1\) zero entries. 
\end{enumerate}
\end{corollary}
\begin{proof}
Clearly, item \ref{it:2} is a direct consequence of item \ref{it:1} and item \ref{it:3} is a direct consequence of item \ref{it:2}. 
To conclude the proof we establish item \ref{it:1}. Lemma \ref{lem:zeroRowSum} tells us that
\begin{equation*}
\sum_{i=1}^m \abs{m_{ki}}(c_{kk}c_{i\ell}-c_{ik}c_{k\ell})=0.
\end{equation*}
Thus, we obtain
\begin{equation}\label{eq:zero1}
\abs{m_{ki}}c_{kk}c_{i \ell}=0
\end{equation}
for all integers \(1 \leq i \leq m\). 
 Since each principal submatrix of \(C\) is the inverse matrix of an M-matrix, cf. \cite[Corollary 3]{johnson1982inverse}, it follows \(c_{kk} \neq 0\). Thus, via Equation \eqref{eq:zero1} we obtain
\(m_{ki}=0\) or \(c_{i\ell}=0\) for all \(i\in\{1, \ldots, m\}\), as desired. 
\end{proof}

Theorem \ref{thm:estiMate} will be established via a density argument.
As it turns out, it will be beneficial to approximate \(C\) by matrices with non-vanishing minors. 
To this end, we need the following genericity condition. 
\begin{definition}[generic matrix]
Let \(m\geq 1\) be an integer and let \(A\in \Mat(m\times m; \R)\) be a matrix. Suppose that \(1 \leq k \leq m\) is an integer
and let \(I,J\subset \{1, \ldots, m\}\) be two subsets such that \(\card{(I)}=\card{(J)}=k\). 

We use the notation \(A[I,J]\in  \Mat(k\times k; \R)\) to denote the matrix that is obtained from \(A\) by keeping the rows of \(A\) that belong to \(I\) and 
the columns of \(A\) that belong to \(J\). 
We say that \(A\) is generic if
\begin{equation*}
\det(A[I,J])\neq 0
\end{equation*}
for all non-empty subsets \(I,J\subset \{1, \ldots, m\}\) with \(\card{(I)}=\card{(J)}\).
\end{definition}
The subsequent lemma demonstrates that being generic is a 'generic property' as used in the context of algebraic geometry. 
\begin{lemma}\label{lem:martin}
Let \(m\geq 1\) be an integer and let \(A\in \Mat(m\times m; \R)\) be a matrix.
The following holds
\begin{enumerate}[label=(\roman*)]
\item if \(A\) is generic, then \(A^{-1}\) is generic as well.
\item the set of generic matrices is open and dense in \( \Mat(m\times m; \R)\). 
\end{enumerate}
\end{lemma}
\begin{proof}
The first item is a direct consequence of  Jacobi's equality, cf. \cite{Jacobi1841},
\begin{equation}\label{eq:Jacobi}
\abs{\det (A^{-1}[I,J])\det(A)}=\abs{\det\left(A\big[[m]\setminus J, [m]\setminus I\big]\right)},
\end{equation}
where \(I,J\subset [m]:=\{1, \ldots, m\}\) with \(\card{(I)}=\card{(J)}\) and \(A[\varnothing, \varnothing]\) is by definition equal to the identity matrix. 
Next, we establish the second item. A matrix \(A\in \Mat(m\times m; \R)\) is generic if and only if
\begin{equation*}
p(A):=\prod_{I,J\subset [m], \abs{I}=\abs{J}} \det (A[I,J]) \neq 0. 
\end{equation*}
Clearly, \(p\) is a non-zero polynomial in the entries of \(A\). It is straightforward to show that the complement of the zero set of a non-zero polynomial \(q\colon \R^N\to \R\) is
an open and dense subset of \(\R^N\), for all \(N\geq 1\). Therefore, the set of generic matrices is an open and dense subset of \( \Mat(m\times m; \R)\), as was to be shown. 
\end{proof}
We proceed with the following lemma, which is the key component in the proof of Theorem \ref{thm:estiMate}. 
\begin{lemma}\label{lem:signPattern}
Let \(m\geq 2\) and let \(A\in \Mat(m\times m; \R)\) be a non-negative matrix. 
If \(A\) is a generic matrix, then for all distinct integers \(1 \leq k,\ell \leq m\) the skew-symmetric matrix \(A^{_{^{^{(k,\ell)}}}}\in \Mat(m\times m; \R)\) 
given by 
\begin{equation*}
a_{ij}^{_{^{^{(k,\ell)}}}}:=a_{ik}a_{j\ell}-a_{jk}a_{i\ell},
\end{equation*}
has the property that each two rows of \(A^{_{^{^{(k,\ell)}}}}\) have a distinct number of positive entries.
\end{lemma}
\begin{proof}
We fix two distinct integers \(1 \leq k,\ell \leq m\). If \(m=2\), then each two rows of  \(A^{_{^{^{(k,\ell)}}}}\) have a distinct number of positive entries, since \(A\) is generic.
Now, suppose that \(m=3\). The matrix  \(A^{_{^{^{(k,\ell)}}}}\) is skew-symmetric; hence, as \(A\) is generic we obtain that \(A^{_{^{^{(k,\ell)}}}}\) can have \(2^{3}\) different sign patterns. If
\begin{equation}\label{eq:cannot}
a_{12}^{_{^{^{_{(k,\ell)}}}}}, a_{23}^{_{^{^{(k,\ell)}}}}, a_{31}^{_{^{^{(k,\ell)}}}} > 0 \quad\textrm{ or }\quad a_{12}^{_{^{^{(k,\ell)}}}}, a_{23}^{_{^{^{(k,\ell)}}}}, a_{31}^{_{^{^{(k,\ell)}}}} <0,
\end{equation}
then each row of  \(A^{_{^{^{(k,\ell)}}}}\) has the same number of positive entries and the statement does not hold.  
For the other 6 sign patterns it is straightforward to check that each row of  \(A^{_{^{^{(k,\ell)}}}}\) has a different number of positive entries.

In the following, we show that \eqref{eq:cannot} cannot occur. For the sake of a contradiction, we suppose \(a_{12}^{_{^{^{(k,\ell)}}}}, a_{23}^{_{^{^{(k,\ell)}}}}, a_{31}^{_{^{^{(k,\ell)}}}} > 0 \). 
Since  \(a_{12}^{_{^{^{(k,\ell)}}}}>0\), we obtain 
\begin{equation}\label{eq:firstConsequence}
a_{1k} > \frac{a_{2k}a_{1\ell}}{a_{2\ell}}.
\end{equation}
Since \(a_{31}^{_{^{^{(k,\ell)}}}} > 0\), we estimate via \eqref{eq:firstConsequence}
\begin{equation}\label{eq:estimateofMinors}
a_{3k}a_{1\ell} > a_{1k}a_{3\ell} > \frac{a_{2k}a_{1\ell}}{a_{2\ell}}a_{3\ell}.
\end{equation}
Thus, \eqref{eq:estimateofMinors} tells us that
\begin{equation*}
a_{3k}a_{2\ell} > a_{2k}a_{3\ell};
\end{equation*}
which contradicts \( a_{23}^{_{^{^{(k,\ell)}}}}>0\). Hence, the case  \(a_{12}^{_{^{^{(k,\ell)}}}}, a_{23}^{_{^{^{(k,\ell)}}}}, a_{31}^{_{^{^{(k,\ell)}}}} > 0 \) cannot occur. The other invalid sign pattern can be treated analogously . Therefore, \eqref{eq:cannot} cannot occur, as claimed. By putting everything together, we conclude that the statement is valid if \(m=3\). 

We proceed by induction. Let \(m\geq 4\) be an integer and suppose that the statement is valid for all \(2 \leq m^\prime < m\). Before we proceed with the proof we introduce some notation. For every matrix \(B\in \Mat(m\times m;\R)\) we denote by \(B_{ij}\in \Mat((m-1)\times (m-1); \R)\) the matrix that is obtained from \(B\) by deleting the \(i\)-th row and the \(j\)-th column of \(B\). Moreover, for all integers \(1 \leq i,j \leq m\) with \(i \neq j\) we set
\begin{equation*}
\begin{split}
&n_i^{+}(B):=\textrm{number of positive entries of the i-th row of } B, \\
&n_{i,j}^{+}(B):=\textrm{number of positive entries of } (b_{i1}, \ldots,  \widehat{b_{ij}}, \ldots, b_{im}). 
\end{split}
\end{equation*}
We use \(\widehat{b_{ij}}\) to indicate that the entry \(b_{ij}\) is omitted. Since the non-negative \((m-1)\times (m-1)\)-matrix \(A_{ij}\) is generic for all \(1 \leq i, j \leq m\), we obtain via the induction hypothesis that each row of \(\left(A^{_{^{^{(k,\ell)}}}}\right)_{ii}\) has a different number of positive entries for all \(1 \leq i \leq m\). For simplicity of notation, we abbreviate \(B:=A^{_{^{^{(k,\ell)}}}}\) for the rest of this proof. We have to show that  each two rows of  \(B\) have a distinct number of positive entries. Let \(p\in \{1, \ldots, m\}\setminus{\{m\}}\) denote the unique integer such  that \(n_{p,m}^{+}(B)=(m-1)-1\), that is, the \(p\)-th row of \(B_{mm}\) has the most positive entries. Suppose that \(b_{pm}>0\). This implies \(n_{p}^{+}(B)=m-1\). Consequently, the \(p\)-th column of \(B\) has no positive entries; hence,
as each two rows of \(B_{pp}\) have a distinct number of positive entries and the number of positive entries of each row of \(B_{pp}\) is strictly smaller than \(m-1\),
we obtain that all rows of \(B\) have a distinct number of positive entries.
Hence, the statement follows if \(b_{pm}>0\). 

Now, we suppose that \( b_{pm} < 0\). This implies \(n_{p}^{+}(B)=m-2\). There is precisely one integer \( q\in \{1, \ldots, m\}\setminus \{p\}\) such that \(n_{q,p}^{+}(B)=(m-1)-1\). Suppose that \(q=m\). Since \(b_{mp} > 0\), we obtain that \(n_{m}^{+}\left(B\right)=m-1\). Thus, we obtain as before via the induction hypothesis that all rows of \(B\) have a distinct number of positive entries. Therefore, the statement follows if  \(q=m\). We are left with the case \( b_{pm} < 0\) and \(q\neq m\). Note that in this case
\begin{equation}\label{eq:contradiction}
n_{p}^{+}\left(B\right)=n_{q}^{+}\left(B\right)=m-2 \textrm{ and } b_{qp}<0.
\end{equation}
As a result, for each integer \(r\in \{1, \ldots, m\}\setminus{\{p,q,m\}}\) both entries \(b_{pr}\) and \(b_{qr}\) are positive.
But via \eqref{eq:contradiction} this implies
\begin{equation*}
n_{p,r}^{+}\left(B\right)=n_{q,r}^{+}\left(B\right)=m-3,
\end{equation*}
for all \(r\in \{1, \ldots, m\}\setminus{\{p,q,m\}}\) which is not possible due to the induction hypothesis.
Therefore, the case \( b_{pm} < 0\) and \(q\neq m\) cannot occur. 

We have considered all cases and thus the statement follows by induction. The lemma follows. 
\end{proof}
We conclude this section with the proof of Theorem \ref{thm:estiMate}. 
\begin{proof}[proof of Theorem \ref{thm:estiMate}]
Fix \(k, \ell\in \{1, \ldots, m\}\) with \(k\neq \ell\). 
Lemma \ref{lem:martin} and a diagonal sequence argument tell us that there is a sequence \(\{C_r\}_{r\geq 1}\), where \(C_r:=(c_{ij}^{(r)})_{_{1 \leq i, j \leq m}}\),  of non-negative generic matrices such that \(C_r\to C\) with \(r\to +\infty\). 
By passing to a subsequence (if necessary) we may assume that the matrices \(C^{_{^{^{(k,\ell)}}}}_r\), defined in Lemma \ref{lem:signPattern},  all have the same sign pattern. For each integer \(r\geq 1\) let \(T_r\in \Mat(m\times m; \R)\) be the matrix given by
\begin{equation*}
t_{ij}^{(r)}:=|m_{ij}^{(r)}|\left( c_{ik}^{(r)}c_{j\ell}^{(r)}-c_{jk}^{(r)}c_{i\ell}^{(r)}\right),
\end{equation*}
where \(M_r:=C_r^{-1}\). Due to the first item in Lemma \ref{lem:martin}, it follows that \(m_{ij}^{(r)}\neq 0\). Thus,  the matrices \(T_r\) and \(C^{_{^{^{(k,\ell)}}}}_r\) have the same sign pattern.

Therefore, by the virtue of Lemma \ref{lem:signPattern}, each row of \(T_r\) has a distinct number of positive entries. Fix an integer \(r \geq 1\). For each integer \(1 \leq p \leq m\) let \(c(p)\) be the unique integer such that
the \(c(p)\)-th row of \(T_r\) has exactly \(m-p\) positive entries. Since all matrices \(T_r\) have the same sign pattern, the definition of \(c\) is
independent of the integer \(r\geq 1\). The map \(c\colon \{1, \ldots, m\}\to \{1, \ldots, m\}\) is a bijection and 
\begin{equation*}\label{eq:TriangularPattern}
\begin{cases}
&t_{c(p) j}^{(r)} < 0 \,\,\, \textrm{ if } \,\,\, j\in \{ c(1), \ldots, c(p-1)\} \,\,\, \\[+0.35em] 
&t_{c(p) j}^{(r)} > 0 \,\,\, \textrm{ if } \,\,\, j\in \{ c(p+1), \ldots, c(m)\}
\end{cases}
\end{equation*}
for all integers \(r\geq 1\). 
Let \(T\in \Mat(m\times m; \R)\) be the matrix given by
\begin{equation*}
t_{ij}:=\abs{m_{ij}}\left( c_{ik}c_{j\ell}-c_{jk}c_{i\ell}\right).
\end{equation*}
Clearly, \(T_r\to T\) with \(r\to +\infty\). As a result, 
\begin{equation}\label{eq:TriangularPatternII}
\begin{cases}
&t_{c(p) j} \leq 0 \,\,\, \textrm{ if } \,\,\, j\in \{ c(1), \ldots, c(p-1)\} \,\,\, \\[+0.35em] 
&t_{c(p) j} \geq 0 \,\,\, \textrm{ if } \,\,\, j\in \{ c(p+1), \ldots, c(m)\}.
\end{cases}
\end{equation}
By Lemma \ref{lem:zeroRowSum} and \eqref{eq:TriangularPatternII} we obtain that
\begin{equation}\label{eq:VeryUsefulIdentity}
\sum_{ j=1}^{p-1} t_{c(j)c(p)}=\sum_{j=p+1}^m t_{c(p)c(j)}
\end{equation}
for all integers \(1 \leq p \leq m\) with \(c(p)\neq k,\ell\), since \(T\) is skew-symmetric. 

In \cite[Theorem 3.1]{markham1972nonnegative}, Markham established that every almost principal minor of \(C\) is non-negative. 
Hence, 
\begin{equation*}
\abs{m_{kj}}\left( c_{kk}c_{j\ell}-c_{jk}c_{k\ell}\right) \geq 0 \,\,\,\textrm{ and }\,\,\, \abs{m_{\ell j}}\left( c_{\ell k}c_{j\ell}-c_{jk}c_{\ell \ell}\right) \leq 0
\end{equation*}
for all integers \(1 \leq j \leq m\). Consequently, we obtain that \(c(1)=k\) and \(c(m)=\ell\). 
For each integer \(2 \leq h \leq m-1\) we compute via \eqref{eq:VeryUsefulIdentity}, 
\begin{equation}\label{eq:induction}
\begin{split}
&\sum_{p=2}^{h} \sum_{j=p+1}^m t_{c(p)c(j)}=\sum_{p=2}^{h} \sum_{ j=1}^{p-1} t_{c(j)c(p)} \\
&=\sum_{j=2}^{h} t_{c(1)c(j)}+\sum_{j=2}^{h-1}\sum_{p=j+1}^{h} t_{c(j)c(p)} \\
&\leq \sum_{j=2}^{h} t_{c(1)c(j)}+\sum_{p=2}^{h-1}\sum_{j=p+1}^{m} t_{c(p)c(j)} .
\end{split}
\end{equation}
Note that
\begin{equation*}\label{eq:theProdigy}
\frac{1}{2}\sum_{i=1}^m \sum_{j=1}^m \abs{m_{ij}\left(c_{ik}c_{jl}-c_{jk}c_{il}\right)}=\sum_{p=1}^m \sum_{j=p+1}^m t_{c(p)c(j)} . 
\end{equation*}
Therefore, by the use of \eqref{eq:induction} we obtain
\begin{equation*}
\begin{split}
&\frac{1}{2}\sum_{i=1}^m \sum_{j=1}^m \abs{m_{ij}\left(c_{ik}c_{jl}-c_{jk}c_{il}\right)}\\
&\leq \sum_{h=2}^{m} \sum_{j=2}^h t_{c(1)c(j)}\leq (m-1)\sum_{j=1}^{m} t_{c(1)c(j)}.
\end{split}
\end{equation*}
Lemma \ref{lem:zeroRowSum} tells us that 
\begin{equation*}
\sum_{j=1}^{m} t_{c(1)c(j)}=c_{k\ell};
\end{equation*}
therefore, the theorem follows. 
\end{proof}


\section{Proof of Theorem \ref{hopefullySolved}}\label{sec:mainProof}
\begin{proof}[Proof of Theorem \ref{hopefullySolved}]
Without loss of generality we may assume (by scaling) that \(\Lip(f)=1\). We set \(I:=X\), \(T:=X\setminus S\) and let the map \(\mathbf{x}\colon I\to H\) be given by the identity.

Let \(G\colon [0,+\infty)\to [0,+\infty)\) denote the function such that \(x=F(\sqrt{G(x)})\) for all real numbers \(x\in [0,+\infty)\). Observe that the function \(G\) is convex, strictly-increasing and \(G(0)=0\). We say that \(\bm{\xi}\colon I \times I \to \R\) \textit{lies above} \(f\)
if there is a map \(\overline{f}\colon X \to \Conv(\mathsf{Im}(f))\) such that \(\overline{f}(s)=f(s)\) for all \(s\in S\) and 
\begin{equation*}
G\left( \norm{\overline{f}\big(\mathbf{x}(i)\big)-\overline{f}\big(\mathbf{x}(j)\big)}_{_H}\right) \leq \bm{\xi}(i,j) \quad \textrm{ for all } i,j \in I.
\end{equation*}
We use \(\Conv\) to denote the closed convex hull. Let \(E_f\subset  \R^{I\times I}\) be the set of all \(\bm{\xi}\in \R^{I\times I}\) that lie above \(f\). Moreover, let
\(\bm{v}\colon I\times I\to \R\) be the map given by
\begin{equation}\label{eq:vTee}
\bm{v}(i,j):= \norm{ \mathbf{x}(i)-\mathbf{x}(j)}_{_H}^2. 
\end{equation}

Suppose that \(L\in[1, +\infty)\) is a real number.  If  \(L\bm{v}\in E_f\), then  the map \(f\) admits a Lipschitz extension \(\overline{f}\colon X\to E\) such that \[\Lip(\overline{f}) \leq \sup_{x>0} \frac{F(\sqrt{L} x)}{F(x)}.\] Indeed,
if \(L\bm{v}\in E_f\), then (by definition) there exists a function \(\overline{f}\colon X \to \Conv(\mathsf{Im}(f))\) such that
\[G\left( \norm{\overline{f}\big(\mathbf{x}(i)\big)-\overline{f}\big(\mathbf{x}(j)\big)}_{_H} \right) \leq L \bm{v}(i,j)\quad \textrm{ for all } i,j \in I;\]
consequently, by applying the function \(F\left(\sqrt{\cdot}\right)\) on both sides, we obtain 
\begin{equation*}
\begin{split}
\norm{\overline{f}\big(\mathbf{x}(i)\big)-\overline{f}\big(\mathbf{x}(j)\big)}_{_H} &\leq F\left(\sqrt{\left(L\norm{ \mathbf{x}(i)-\mathbf{x}(j)}_{_H}^2\right)}\right) \\
&\leq \sup_{x>0} \frac{F(\sqrt{L} x)}{F(x)}\, F\big(\norm{\mathbf{x}(i)-\mathbf{x}(j)}_{_H}\big)
\end{split}
\end{equation*}
for all \(i,j \in I\). Since \(X\subset F[H]\) the map \(\overline{f}\) is a Lipschitz extension of \(f\) such that \(\Lip(\overline{f})\) has the desired upper bound. 
Thus, to prove the theorem it suffices to show that if \(L\geq (m+1)\), then \(L\bm{v}\in E_f\). 

To this end, we suppose that \(L\bm{v}\notin E_f\) and we show that \(L < (m+1) \). Since the function \(G\) is strictly-increasing and convex, the set \(E_f\) is closed and convex; thus, by the hyperplane separation theorem we obtain a real number \(\epsilon >0 \) and a non-zero vector \(\bm{\lambda}\in \R^{I\times I}\) such that
\begin{equation}\label{Eq:HyperplaneSeparation}
\langle L\bm{v} , \bm{\lambda} \rangle_{\R^{I\times I }} + \epsilon <  \langle \bm{\xi} , \bm{\lambda} \rangle_{\R^{I\times I}}\quad \textrm{ for all } \bm{\xi} \in E_f. 
\end{equation}

We claim that each entry of \(\bm{\lambda}\) is non-negative. Indeed, if \(\bm{\xi}\in E_f\), then the point \((\xi_1, \ldots, \xi_{k-1}, c \xi_k, \xi_{k+1}, \ldots, \xi_N)\), where \(N:=\card(I\times I)\), is contained in \( E_f\) for all integers \(1 \leq k \leq N\) and real numbers \(c\in [1, +\infty)\). Hence, a simple scaling argument implies that the \(k\)-th entry of \(\bm{\lambda}\) is non-negative for each integer \(1 \leq k \leq N\), as claimed. 

In the following, we estimate \( \langle L\bm{v} , \bm{\lambda} \rangle_{\R^{I \times I }}\) from below. We may assume that \(\bm{\lambda}\) is symmetric. By adjusting \(\epsilon  >0 \) if necessary, we may assume that \(\sum_{k \in S} \lambda_{i k} \neq 0\) for all \(i\in T\).
Let the matrix \(M:=M(\bm{\lambda}, T)\) be given as in \eqref{eq:MMatrix}. Since each entry of the vector \(\bm{\lambda}\) is non-negative and \(\sum_{k \in S} \lambda_{i k} \neq 0\) for all \(i\in T\), the matrix \(M(\bm{\lambda}, T)\) is non-singular. We set \(C:=M^{-1}\). 
Proposition \ref{prop:quadratic} tells us that 
\begin{equation}\label{eq:Equality77}
\mathsf{m}:=\mathsf{m}(\mathbf{x}, \bm{\lambda}, \id, T)= \sum_{r\in S} \sum_{s\in S} \bm{\eta}(r,s) \norm{\mathbf{x}(r)-\mathbf{x}(s)}^{2}_{_H},
\end{equation}
where \(\bm{\eta}\colon I\times I\to \R\) is given by
\[\bm{\eta}(r,s):=\lambda_{rs}+\sum_{i\in T}\sum_{j\in T} \lambda_{ir} c_{ij} \lambda_{js}.\]
 Clearly,
\begin{equation}\label{eq:Lowe}
L \mathsf{m} \leq \langle L\bm{v} , \bm{\lambda} \rangle_{\R^{I \times I}}.
\end{equation}
Next, we estimate \(\langle L\bm{v} , \bm{\lambda} \rangle_{\R^{I\times I}}\) from above. We set 
\begin{equation*}
\bar{\bm{\lambda}}_i:=\frac{1}{\norm{\bm{\lambda}_i}_{_1} }\bm{\lambda}_i \in \Delta^{\card(S)-1}
\end{equation*}
for each \(i\in T\), where \(\bm{\lambda}_i:=(\lambda_{ik})_{k\in S} \).
By \eqref{eq:SumToOne},
\begin{equation}\label{eq:almostDone}
\sum_{j\in T} c_{ij} \norm{\bm{\lambda}_i}_{_1}=\sum_{j\in T} c_{ij} \sum_{k\in S} \lambda_{jk}=1
\end{equation}
for all \(i\in T\). 
For each \(i\in T\) we define
\begin{equation*}
w_{i}:=\sum_{j\in T} c_{ij}\left(\sum_{k\in S} \lambda_{jk} \right) y_{\bar{\bm{\lambda}}_j}, \,\,\textrm{ where  }\, y_{\bar{\bm{\lambda}}_j}=\sum_{r\in S} \bar{\lambda}_{jr} f(r).
\end{equation*}
 Using \eqref{eq:almostDone} we obtain \(w_i\in \Conv(\mathsf{Im}(f))\) for all \(i\in T\). 
Equation \eqref{Eq:HyperplaneSeparation} tells us that 
\begin{equation}\label{eq:fromabove}
\langle L\bm{v} , \bm{\lambda} \rangle_{\R^{I\times I}} < A+B+C;
\end{equation}
where,
\begin{equation*}
\begin{split}
&A:=2\sum_{i\in T} \sum_{r\in S} \lambda_{ir}G\left(\norm{f(r)-w_i}_{_E}\right), \\
& B:=\sum_{i\in T}\sum_{j\in T} \lambda_{ij} G\left(\norm{w_i-w_j}_{_E}\right), \\
&C:=\sum_{r\in S} \sum_{s\in S} \lambda_{rs}  G\left(\norm{f(r)-f(s)}_{_E}\right).
\end{split}
\end{equation*}
By convexity of the strictly-increasing function \(G\) and the use of \eqref{eq:almostDone}, we estimate
\begin{equation*}
\begin{split}
&\, A+C \\
&\leq 2\, \sum_{i\in T}\sum_{r\in S}\sum_{j\in T} \lambda_{ir}c_{ij}\norm{\bm{\lambda}_j}_{_1}\,G\left(\norm{f(r)- y_{\bar{\bm{\lambda}}_j}}_{_E}\right)+C \\
& \leq  2\, \sum_{r\in S} \sum_{s\in S} \bm{\eta}(r,s)\, G\left(\norm{f(r)-f(s)}_{_E}\right).
\end{split}
\end{equation*}
Thus, if 
\begin{equation}\label{eq:last}
B=\sum_{i\in T}\sum_{j\in T} \lambda_{ij} G\left(\norm{w_i-w_j}_{_E}\right) \leq (m-1)\, \sum_{r\in S} \sum_{s\in S} \bm{\eta}(r,s) G\left(\norm{f(r)-f(s)}_{_E}\right),
\end{equation}
then we obtain via \eqref{eq:fromabove} and \eqref{eq:Lowe} that
\begin{equation*}
L\mathsf{m} < (m+1) \, \sum_{r\in S} \sum_{s\in S} \bm{\eta}(r,s) G\left(\norm{f(r)-f(s)}_{_E}\right).
\end{equation*}
Since 
\[\norm{f(r)-f(s)}_{_E} \leq F\big( \sqrt{\norm{r-s}_{_H}^2}\big) \quad \textrm{ for all } r,s\in S,\]
it follows
\[G(\norm{f(r)-f(s)}_{_E})\leq \norm{ r-s }_{_H}^2  \quad \textrm{ for all } r,s\in S; \]
as a result, we obtain
\[L\mathsf{m} < (m+1) \mathsf{m}.\]
By virtue of Corollary \ref{Cor:M-mat} every entry of  the matrix \(C\) is positive, hence \(\mathsf{m}>0\) and consequently \(L  < m+1\). Thus, to conclude the proof we are left to establish the estimate \eqref{eq:last}. 
It is readily verified that
\begin{equation*}
w_i-w_j=\frac{1}{2}\sum_{k\in T}\sum_{\ell\in T} \norm{\bm{\lambda}_k}_{_1}\norm{\bm{\lambda}_\ell}_{_1}\left(c_{j\ell}c_{ik}- c_{i\ell}c_{jk}\right)\left(y_{\bar{\bm{\lambda}}_k}-y_{\bar{\bm{\lambda}}_\ell} \right).
\end{equation*}
Since
\begin{equation*}
\frac{1}{2}\sum_{k\in T} \sum_{\ell\in T} \norm{\bm{\lambda}_k}_{_1}\norm{\bm{\lambda}_\ell}_{_1}\abs{c_{j\ell}c_{ik}- c_{i\ell}c_{jk}} \leq \sum_{k\in T} \abs{c_{ik}}\norm{\bm{\lambda}_k}_{_1}\sum_{\ell\in T} |c_{j\ell}| \,\norm{\bm{\lambda}_\ell}_{_1}=1, 
\end{equation*}
we can use the triangle inequality, the convexity of the strictly-increasing map \(G\) and \(G(0)=0\) to estimate
\begin{equation}\label{eq:intermediate}
\begin{split}
&B=\sum_{i\in T}\sum_{j\in T} \lambda_{ij} G\left(\norm{w_i-w_j}_{_E}\right) \\
&\leq \sum_{i\in T}\sum_{j\in T} \lambda_{ij}\, \frac{1}{2} \sum_{k\in T} \sum_{\ell\in T} \norm{\bm{\lambda}_k}_{_1}\norm{\bm{\lambda}_\ell}_{_1}\,\abs{c_{j\ell}c_{ik}- c_{i\ell}c_{jk}} \,\, G\left(\norm{ y_{\bar{\bm{\lambda}}_k}-y_{\bar{\bm{\lambda}}_\ell}}_{_E}\right) \\
&=\sum_{k\in T} \sum_{\ell\in T} \norm{\bm{\lambda}_k}_{_1}\norm{\bm{\lambda}_\ell}_{_1} \left( \frac{1}{2} \sum_{i\in T} \sum_{j\in T} \lambda_{ij}\abs{c_{ik}c_{j\ell}- c_{jk}c_{i\ell}} \right)G\left(\norm{ y_{\bar{\bm{\lambda}}_k}-y_{\bar{\bm{\lambda}}_\ell}}_{_E}\right).
\end{split}
\end{equation}
As pointed out in the beginning of Section \ref{sec:betterName}, \(M(\bm{\lambda}, T)\) is a symmetric M-matrix. Hence, we may invoke Theorem \ref{thm:estiMate} and obtain
\begin{equation*}
\frac{1}{2} \sum_{i\in T} \sum_{j\in T} \lambda_{ij}\abs{c_{ik}c_{j\ell}- c_{jk}c_{i\ell}} \leq (m-1)c_{k\ell}
\end{equation*}
for all distinct \(k, \ell\in T\). Using \eqref{eq:intermediate} we deduce
\begin{equation*}
\begin{split}
&\sum_{i\in T}\sum_{j\in T} \lambda_{ij} G\left(\norm{w_i-w_j}_{_E}\right)  \\
&\leq \left(m-1\right)\sum_{k\in T} \sum_{\ell\in T} \norm{\bm{\lambda}_k}_{_1}\norm{\bm{\lambda}_\ell}_{_1}  c_{k\ell}\,\, G\left(\norm{ y_{\bar{\bm{\lambda}}_k}-y_{\bar{\bm{\lambda}}_\ell}}_{_E}\right).
\end{split}
\end{equation*}
By convexity,
\begin{equation*}
G\left(\norm{ y_{\bar{\bm{\lambda}}_k}-y_{\bar{\bm{\lambda}}_\ell}}_{_E}\right) \leq \sum_{r\in S}\sum_{s\in S} \bar{\lambda}_{kr}\bar{\lambda}_{\ell r}G\big(\norm{f(r)-f(s)}_{_E}\big);
\end{equation*}
thereby, the desired estimate \eqref{eq:last} follows, as was left to show. This completes the proof. 
\end{proof}


\subsection{Acknowledgements} I am indebted to Urs Lang, for his suggestions enabled me to simplify the proofs of the main results considerably. Further, I am thankful to Assaf Naor for pointing out to me the simple bound of \(e^m(X,Y)\) in terms of \(e_n(X,Y)\). I am thankful to the anonymous reviewer for helpful comments that shortened the proof of Lemma \ref{Lem:Lower} and for drawing my attention to additional literature. Finally,
I want to thank Yannick Krifka and Martin Stoller, who read earlier draft versions of this paper.


\bibliographystyle{plain}
\bibliography{refs}
\noindent
\textsc{\small{Mathematik Departement, ETH Zürich, Rämistrasse 101, 8092 Zürich, Schweiz}}\\
\textit{E-mail address:}{\textsf{ giuliano.basso@math.ethz.ch}}\\

\end{document}